\documentclass[a4paper,11pt]{article}
\usepackage[top=3.0cm, bottom=3.0cm, inner=3.0cm, outer=3.0cm, includefoot]{geometry}

\usepackage[utf8]{inputenc}
\usepackage[T1]{fontenc}
\usepackage{authblk}
\usepackage{verbatim}
\usepackage{multirow}
\usepackage{multicol}
\usepackage{hyperref}
\usepackage{amssymb}
\usepackage{amsmath}
\usepackage{mathtools,bbm}
\usepackage{graphicx,tikz}
\usepackage{amsthm}
\usepackage{caption,subcaption}
\usepackage{float}
\usepackage{color}
\usepackage{enumerate}

\usepackage{cancel}

\newcommand{\eChar}{\begin{enumerate}[(i)]}
\newcommand{\eCharR}{\begin{enumerate}[(a)]}
\newcommand{\eBr}{\begin{enumerate}[(1)]}

\newcommand{\Abstract}

\theoremstyle{plain}

\newtheorem{lemma}{Lemma}[section]
\newtheorem{theorem}[lemma]{Theorem}

\newtheorem{corollary}[lemma]{Corollary}
\theoremstyle{definition}

\newtheorem{definition}[lemma]{Definition}
\newtheorem{remark}[lemma]{Remark}

\newtheorem{problem}[lemma]{Problem}

\numberwithin{equation}{section}

\title
{
Diameter bounds for distance-regular graphs via long-scale Ollivier Ricci curvature

}

\author[1]{Kaizhe Chen\thanks{Email: ckz22000259@mail.ustc.edu.cn}}
\author[2]{Shiping Liu\thanks{Email: spliu@ustc.edu.cn}}
\affil[1]{School of Gifted Young, University of Science and Technology of China, Hefei}
\affil[2]{School of Mathematical Sciences, University of Science and Technology of China, Hefei}

\date{ }

\begin{document}

\maketitle
\thispagestyle{plain}
\begin{abstract}
In this paper, we derive new sharp diameter bounds for distance regular graphs, which better answer a problem raised by Neumaier and Penji\' c in many cases. Our proof is built upon a relation between the diameter and long-scale Ollivier Ricci curvature of a graph, which can be considered as an improvement of the discrete Bonnet-Myers theorem. Our method further leads to significant improvements of existing diameter bounds for amply regular graphs and $(s,c,a,k)$-graphs.
\end{abstract}

\section{Introduction}
Distance-regular graphs play an important role in algebraic combinatorics due to their close and deep relation to design theory, coding theory, finite and Euclidean geometry, and group theory \cite{BCN89, DKT}. Bounding the diameter of a distance-regular graph in terms of its intersection numbers is a very important problem which has attracted lots of attention \cite{BDKM,BHK,HLX24,I83,Mulder79,NP22JCTB,NP22,Smith74,T82,T83}. In \cite[Problem 1.1]{NP22}, Neumaier and Penji\' c raised a question asking for diameter bounds in terms of a small initial part of the intersection array. 

\begin{problem}[{\cite{NP22}}]\label{prob:NP}
    Let $G$ denote a distance-regular graph, and assume that we only know the first $ q + 2$ elements $b_i$ and $c_i$ of intersection array
    \begin{align}\label{intersection array}
        \{ b_0,b_1,...,b_q,b_{q+1},...; c_1,c_2,...,c_{q+1},c_{q+2},...\},
    \end{align}
i.e., assume that we don’t know intersection numbers $b_{q+2},..., b_{d-1}$ and $c_{q+3}, . . . , c_{d}$. Use the numbers given in \eqref{intersection array} to give an upper bound for the diameter of $G$.
\end{problem}



For a distance-regular graph $G$ of diameter $d$ and valency $k$, we denote its intersection array by $\{b_0,b_1,\ldots, b_{d-1}; c_1,c_2,\ldots, c_d\}$ (see Section \ref{Distance-regular graph} for definitions). We further denote $a_i:=k-b_i-c_i$, $0\le i\le d$, where we use the notation $c_0=b_d=0$. 
In \cite{NP22}, Neumaier and Penji\'c give the following upper bound for the diameter of $G$.

\begin{theorem}[\cite{NP22}]\label{NP}
    Let $G$ denote a distance-regular graph of diameter $d$, valency $k\ge 3$ and let $q$ be an integer with $2 \le q \le d-1$. If $c_{q+1}> c_q$ and $a_q \le c_{q+1}-c_q$ then
    \begin{align}\label{eqNP}
        d\le \left(\left\lfloor\frac{k-c_{q+1}-1}{c_q}\right\rfloor +2\right)q+1.
    \end{align}
\end{theorem}


Note that Theorem \ref{NP} does not use all the numbers in \eqref{intersection array}. Moreover, the inequality \eqref{eqNP} is possible to take equality only when $d-1$ is a multiple of $q$. In this paper, we derive the following diameter bound for distance-regular graphs using all the numbers in \eqref{intersection array}, which is possible to take equality even if $d-1$ is not a multiple of $q$.

\begin{theorem}\label{diameter}
    Let $G$ be a distance-regular graph of diameter $d$ and valency $k$. Let $q$ be an integer with $1\le q\le d-1$ such that $a_{q-1}=0$, $c_{q+1}>c_q$ and $c_{q+1}\ge a_{q}$. Then, for any $0\le p\le d$, we have
    \begin{align}\label{eq:diameter}
       d\le \max_{0\le r\le q-1} \left\{\left\lfloor\frac{b_p-c_p+b_r-c_r}{2c_q+M}\right\rfloor q +p+r\right\}, 
    \end{align}
    where $$M=\left\lceil \frac{a_q(c_{q+1}-a_{q})}{c_{q+1}-c_q} \right\rceil.$$
\end{theorem}
\begin{remark}
Our Theorem \ref{diameter} better answers Problem \ref{prob:NP} in many cases. We provide three examples below. 
\begin{itemize}
    \item[(i)] If we know that $\{22, 21, 20, 3,...; 1, 2, 3, 20,...\}$ are the first $8$ numbers of the intersection array of a distance-regular graph $G$, then $b_4\le k-c_4=2$. Applying Theorem \ref{diameter} with $q=3$ and $p=4$ shows that the diameter $d$ of $G$ is at most $6$, which is sharp for the coset graph of the shortened binary Golay code with intersection array $\{22, 21, 20, 3,2,1; 1, 2, 3, 20,21,22\}$. Note that Theorem \ref{NP} can only tell $d\le 7$.
    \item[(ii)] If we know that $\{5, 4, 1,...; 1,1,4,...\}$ are the first $6$ numbers of the intersection array of a distance-regular graph $G$, then $b_3\le k-c_3=1$. Applying Theorem \ref{diameter} with $q=2$ and $p=3$ shows that the diameter $d$ of $G$ is at most $4$, which is sharp for the Wells graph with intersection array $\{5, 4, 1,1; 1,1,4,5\}$. Note that Theorem \ref{NP} can only tell $d\le 5$.
    \item[(iii)] Let $\{21, 20, 16, 6,2,...; 1, 2, 6, 16,...\}$ be the first $9$ numbers of the intersection array of a distance-regular graph $G$. Applying Theorem \ref{diameter} with $q=2$ and $p=4$ shows that the diameter $d$ of $G$ is at most $6$, which is sharp for the coset graph of the once shortened and once truncated binary Golay code with intersection array $\{21, 20, 16, 6,2,1,0; 1, 2, 6, 16,20,21\}$. Note that Theorem \ref{NP} can only tell $d\le 7$.
\end{itemize}
\end{remark}

If we only know the values of $a_{q-1},a_q,c_q$ and $c_{q+1}$, we still get a diameter bound as follows.

\begin{theorem}\label{diameter2}
    Let $G$ be a distance-regular graph of diameter $d$ and valency $k$. Let $q$ be an integer with $1\le q\le d-1$ such that $a_{q-1}=0$, $c_{q+1}>c_q$ and $c_{q+1}\ge a_{q}$. Then, for any $0\le p\le d$, we have
    \begin{align}\label{eq:diameter2}
       d\le 2p-1+ \max \left\{0, \left(\left\lfloor\frac{2(b_p-c_p)}{2c_q+M}\right\rfloor+1\right)q \right\}, 
    \end{align}
    where $$M=\left\lceil \frac{a_q(c_{q+1}-a_{q})}{c_{q+1}-c_q} \right\rceil.$$
\end{theorem}

We prove Theorems \ref{diameter} and \ref{diameter2} via establishing a relation between the diameter and long-scale Ollivier Ricci curvature \cite{CK19, O09} of a graph (see Theorem \ref{Wasserstein}), which can be considered as an improvement of the discrete Bonnet-Myers theorem via Ollivier/Lin-Lu-Yau curvature \cite{LLY11,O09}. 
Then our proofs are built upon estimating the long-scale Ollivier Ricci curvature of distance-regular graphs (see Theorems \ref{WassersteinJump} and \ref{WassersteinDRG}).

Our method further leads to the following diameter bounds for amply regular graphs and $(s,c,a,k)$-graphs (see Definitions \ref{ARGdefinition} and \ref{scakdefinition}), which significantly improve the corresponding previous results. The $(s,c,a,k)$-graphs are introduced by Terwilliger \cite{T83} as a generalization of distance-regular graphs. In case $s=2$, an $(s,c,a,k)$-graph is amply regular. The following Corollaries \ref{ARG} and \ref{scak} can be considered as extensions of Theorem \ref{diameter2}.

\begin{corollary}\label{ARG}
Let $G$ be a connected amply regular graph of diameter $d\ge 4$ with parameters $(v,k,\lambda,\mu)$, where $1\ne \mu\ge\lambda$. Then
    \begin{align}\label{eq:ARGdiameter}
    d\le  \left\lfloor\frac{2(k-2\mu)}{2+\left\lceil \frac{\lambda(\mu-\lambda)}{\mu-1} \right\rceil}\right\rfloor+4. 
    \end{align}
\end{corollary}

\begin{remark}
    Let $G$ be a connected amply regular graph of diameter $d\ge 4$ with parameters $(v,k,\lambda,\mu)$, where $1\ne \mu>\lambda$. Brouwer, Cohen and Neumaier \cite[Corollary 1.9.2]{BCN89} prove that 
\begin{equation}\label{eq:BCN}
    d\leq  k-2\mu+4.
\end{equation}
Huang, Liu and Xia \cite{HLX24} prove that \begin{equation}\label{eq:HLX24}
    d\leq \left\lfloor\frac{2k}{3}\right\rfloor.
\end{equation}
Note that Corollary \ref{ARG} significantly improves both \eqref{eq:BCN} and \eqref{eq:HLX24}.
\end{remark}

\begin{corollary}\label{scak}
Let $G$ be an $(s,c,a,k)$-graph with $a\le c$ and diameter $d$. Then
    \begin{align}\label{eq:scakdiameter}
    d\le \max \left\{2s, (s-1)\left( \left\lfloor\frac{2(\delta-2c)}{2+M}\right\rfloor +3 \right)+2 \right\}, 
    \end{align}
    where $\delta$ is the minimum valency of $G$ and $M=\left\lceil \dfrac{a(c-a)}{c-1} \right\rceil$.
\end{corollary}

\begin{remark}
    Let $G$ be an $(s,c,a,k)$-graph with $a\le c$, $c>2$ and diameter $d$. Terwilliger \cite[Theorem 2]{T83} prove that 
\begin{equation}\label{eq:T83}
    d\le \max \left\{3s-1, (s-1)\left( \frac{2ck-2c}{3c-2} -2c+5 \right)+2 \right\}.
\end{equation}
    Note that we use the minimum valency $\delta$ instead of the maximum valency $k$ in \eqref{eq:scakdiameter}. In addition, the coefficient of $\delta$ in \eqref{eq:scakdiameter} is at most $\frac{2}{3}$ (when $M\ne 0$) and the coefficient of $k$ in \eqref{eq:T83} is $\frac{2c}{3c-2}>\frac{2}{3}$, which implies that \eqref{eq:scakdiameter} improves \eqref{eq:T83} for large $k$.
\end{remark}

The rest of the paper is organized as follows.
In Section \ref{Preliminaries}, we recall definitions and properties about distance-regular graphs and perfect matchings. In Section \ref{Wasserstein distance}, we recall the concept of Wasserstein distance and establish a relation between the diameter and the long-scale Ollivier Ricci curvature of a graph. In Section 4, we will prove Theorems \ref{diameter} and \ref{diameter2}. In Section 5, we will prove Corollaries \ref{ARG} and \ref{scak}.

\section{Preliminaries}\label{Preliminaries}
\subsection{Distance-regular graph}\label{Distance-regular graph}
Let $G=(V,E)$ be a graph with vertex set $V$ and edge set $E$. For any $x\in V$, let $k_x$ be the valency of $x$. For any two vertices $x$ and $y$ in $V$, we denote by ${\rm d}(x,y)$ the distance between them. We write $x\sim y$ if $x$ and $y$ are adjacent.

Let $G=(V,E)$ be a graph with diameter $d$. For a vertex $x\in V$ and any non-negative integer $h\le d$, let $S_h(x)$ denote the subset of vertices in $V$ that are at distance $h$ from $x$. We use the notations that $S_{-1}(x) = S_{d+1}(x) := \emptyset$. For any two vertices $x$ and $y$ in $V$ at distance $h$, let
$$A_h(x,y):=S_{h}(x)\cap S_{1}(y),\ B_h(x,y):=S_{h+1}(x)\cap S_{1}(y),\ C_h(x,y):=S_{h-1}(x)\cap S_{1}(y).$$

We say $G=(V,E)$ is \textit{regular with valency $k$} if each vertex in $G$ has exactly $k$ neighbors. A graph $G$ is called \textit{distance-regular} if there are integers $b_i$, $c_i$, $0 \le i \le d$ which satisfy $b_i = |B_i(x, y)|$ and $c_i = |C_i(x, y)|$ for any two vertices $x$ and $y$ in $V$ at distance $i$. Clearly such a graph is regular of valency $k := b_0, b_d = c_0 = 0, c_1 = 1$ and $$a_i:= |A_i(x, y)| = k-b_i-c_i,\ 0 \le i \le d.$$

The array $\{b_0,b_1,...,b_{d-1}; c_1,c_2,...,c_d\}$ is called the \textit{intersection array} of $G$. The following properties of intersection arrays are well-known.
\begin{lemma}[{\cite[Proposition 4.1.6]{BCN89}}]
    Let $G$ be a distance-regular graph of diameter $d\ge2$, valency $k$ and intersection numbers $c_i, a_i, b_i, 0 \le i \le d$. The following hold.
    \begin{itemize}
    \item [\rm{(i)}] $k=b_0>b_1\ge b_2\ge\dots\ge b_d = 0$,
    \item [\rm{(ii)}] $c_0< 1=c_1\le c_2\le\dots \le c_d\le k$.
    \end{itemize}
\end{lemma}

For more information about distance-regular graphs,  we refer the reader to \cite{BCN89}.

\subsection{Perfect matching}
Let us recall the definition of a \emph{perfect matching}.
\begin{definition}
    Let $G$ be a graph. A set $\mathcal M$ of pairwise non-incident edges is called a \emph{matching}. 
    Each vertex adjacent to an edge of $\mathcal M$ is said to be \emph{covered} by $\mathcal M$. 
    A matching $\mathcal M$ is called a \emph{perfect matching} if it covers every vertex of the graph.
\end{definition}

The following K\"onig's theorem \cite{K16} is a key tool in estimating the (long-scale) Ollivier Ricci curvature.

\begin{theorem}[König's theorem]\label{konig}A bipartite graph $G$ can be decomposed into $k$ edge-disjoint perfect matchings if and only if $G$ is $k$-regular.\end{theorem}

Notice that in Theorem \ref{konig}, the bipartite graph is allowed to have multiple edges.

\section{Wasserstein distance and diameter}\label{Wasserstein distance}
In this section, we will prove Theorem \ref{Wasserstein} which relates various Wasserstein distances to diameter bounds of graphs. This provides the basic philosophy of our method. 

We first recall the definition of Wasserstein distance.

 \begin{definition}[Wasserstein distance]
     Let $G=(V,E)$ be a graph, $\mu_1$ and $\mu_2$ be two probability measures on $V$. The Wasserstein distance $W_1(\mu_1, \mu_2)$ between $\mu_1$ and $\mu_2$ is defined as
     \[W_1(\mu_1,\mu_2)=\inf_{\pi}\sum_{y\in V}\sum_{x\in V}{\rm d}(x,y)\pi(x,y),\]
     where the infimum is taken over all maps $\pi: V\times V\to [0,1]$ satisfying
     \[\mu_1(x)=\sum_{y\in V}\pi(x,y),\,\,\mu_2(y)=\sum_{x\in V}\pi(x,y).\] Such a map is called a transport plan.
\end{definition}

For any $\varepsilon\in [0,1]$, let $\mu_x^\varepsilon$ be the probability measure defined as follows:
     \[\mu_x^\varepsilon(y)=\left\{
        \begin{array}{ll}
        \varepsilon, & \hbox{if $y=x$;} \\
        \frac{1-\varepsilon}{k_x}, & \hbox{if $y\sim x$;} \\
        0, & \hbox{otherwise.}
        \end{array}\right.
     \]

We prove the following diameter estimate using Wasserstein distance, which is an improvement of the discrete Bonnet-Myers theorem via Ollivier/Lin-Lu-Yau curvature \cite{LLY11,O09}. 
\begin{theorem}\label{Wasserstein}
    Let $G$ be a connected graph and $0\le\varepsilon <1$ be a constant. Let $q> 0$ and $p\ge 0$ be two integers. Let $C_1>0$ and $C_2$ be two constants such that \begin{itemize}
        \item[\rm (1)] $W_1(\mu_x^\varepsilon, \mu_y^\varepsilon) \le q-C_1$ for any two vertices $x,y$ with ${\rm d}(x,y)=q$,
        \item[\rm (2)] $W_1(\mu_x^1, \mu_y^\varepsilon)\le p+C_2$ for any two vertices $x,y$ with ${\rm d}(x,y)=p$.
    \end{itemize}
    Then $G$ is finite with diameter $d$ satisfying
    \begin{align}\label{eq:Wasserstein}
        d\le 2p-1+ \max\left\{0,\left(\left\lfloor\frac{2C_2}{C_1} \right\rfloor +1\right)q\right\}.
    \end{align}
\end{theorem}

\begin{proof}
    If \eqref{eq:Wasserstein} does not hold, there exist
    two vertices $x$ and $y$ with ${\rm d}(x,y)=D$ such that $D=2p+lq$, where $l$ is an integer satisfying
    \begin{align}\label{l}
        l\ge 0\ {\rm and}\ l\ge \left\lfloor\frac{2C_2}{C_1} \right\rfloor +1.
    \end{align}
    Let $L$ be a path of length $D$ connecting $x$ and $y$. On the path $L$, there is a sequence of vertices $x_0,x_1,...,x_{l}$ such that ${\rm d}(x,x_0)={\rm d}(x_{l} ,y)=p$ and ${\rm d}(x_{i-1},x_i)=q$ for $1\le i\le l$.

    Note that $W_1(\mu_x^1,\mu_y^1)=D$. It follows by the triangle inequality that
    \begin{align}\notag
    D=W_1\left(\mu_x^1,\mu_y^1\right) &\le W_1\left(\mu_x^1,\mu_{x_0}^{\varepsilon}\right) + \sum_{i=1}^l W_1\left(\mu_{x_{i-1}}^\varepsilon,\mu_{x_{i}}^\varepsilon\right) + W_1\left(\mu_{x_l}^\varepsilon, \mu_{y}^1\right)
    \\ \notag &\le 2(p+C_2)+l(q-C_1).
    \end{align}
    That is $lC_1\le 2C_2$, which is contradictory to \eqref{l}.
\end{proof}

\begin{remark}
    The Wasserstein distance and the Ollivier Ricci curvature are directly related. Let $G$ be a  connected graph. For $p\in [0, 1]$, the $p$-Ollivier Ricci curvature of two vertices $x, y$ in $G$ is defined as
    $$\kappa_p(x,y)=1-\frac{W_1(\mu_x^p,\mu_y^p)}{{\rm d}(x,y)}.$$
    In particular, we call the curvature $\kappa_p(x,y)$ "long scale" when ${\rm d}(x, y)\ge 2$.
    The concept of Ollivier Ricci curvature was introduced by Ollivier in \cite{O09}, and the long-scale Ollivier Ricci curvature was futher studied in \cite{CK19}.
\end{remark}

\section{Proofs of Theorems \ref{diameter} and \ref{diameter2}}
In this section, we prove Theorems \ref{diameter} and \ref{diameter2} via (the philosophy of) Theorem \ref{Wasserstein}. For that purpose, we first show two Wasserstein distance estimates, stated as Theorems \ref{WassersteinJump} and \ref{WassersteinDRG} below.

\begin{theorem}\label{WassersteinJump}
    Let $G$ be a connected graph and $0\le\varepsilon <1$ be a constant. Let $x$ and $y$ be two vertices in $G$ at distance $p$, then $$W\left(\mu_x^{1}, \mu_y^{\varepsilon}\right)\le p+\frac{(1-\varepsilon)(|B_p(x,y)|-|C_p(x,y)|)}{k_y}.$$
\end{theorem}
\begin{proof}
    We consider the following particular transport plan $\pi_0$ from $\mu_x^{1}$ to $\mu_y^{\varepsilon}$:
\begin{center}
$\pi_0(v,u)=\begin{cases}
\varepsilon, &{\rm if}\ v=x, u=y;\\
\frac{1-\varepsilon}{k_y}, &{\rm if}\ v=x, u\sim y;\\
0, &{\rm otherwise}.
\end{cases}$
\end{center}
In $S_1(y)$, there are $|A_p(x,y)|$ vertices at distance $p$ from $x$, $|B_p(x,y)|$ vertices at distance $p+1$ from $x$ and $|C_p(x,y)|$ vertices at distance $p-1$ from $x$. Thus,
\begin{align}\notag
    W\left(\mu_x^{1}, \mu_y^{\varepsilon}\right) &\le \sum_{v\in V}\sum_{u\in V}{\rm d}(v,u)\pi_0(v,u)\\ \notag &\le \varepsilon p + \frac{1-\varepsilon}{k_y}\left(|A_p(x,y)|p+|B_p(x,y)|(p+1)+|C_p(x,y)|(p-1)\right) \\ \notag &=p+\frac{(1-\varepsilon)(|B_p(x,y)|-|C_p(x,y)|)}{k_y},
\end{align}
completing the proof.
\end{proof}

\begin{theorem}\label{WassersteinDRG}
    Let $G$ be a distance-regular graph of diameter $d$ and valency $k$. Let $q$ be an integer with $1\le q\le d-1$ such that $a_{q-1}=0$, $c_{q+1}>c_q$ and $c_{q+1}\ge a_{q}$. Let $x$ and $y$ be two vertices in $G$ with ${\rm d}(x,y)=q$, then $$W\left(\mu_x^{\frac{1}{k+1}}, \mu_y^{\frac{1}{k+1}}\right)\le q-\frac{2c_q + M}{k+1},$$where
    $$M=\left\lceil \frac{a_q(c_{q+1}-a_{q})}{c_{q+1}-c_q} \right\rceil.$$
\end{theorem}

By the definition of a distance-regular graph, we have
$$|A_q(y,x)|=|A_q(x,y)|=a_q,|B_q(y,x)|=|B_q(x,y)|=b_q\ {\rm and}\ |C_q(y,x)|= |C_q(x,y)|=c_q.$$
We first prove the following two lemmas.

\begin{lemma}
If $q\ge 2$, then there exists a bijection $\phi$ from $C_q(y,x)$ to $C_q(x,y)$ such that ${\rm d}(v,\phi(v))=q-2$ for every $v\in C_q(y,x)$.
\end{lemma}
\begin{proof}
For any $v\in C_q(y,x)$, we have ${\rm d}(v,y)=q-1$. We claim that $C_{q-1}(v,y)\subset C_q(x,y)$. Indeed, for any $u\in C_{q-1}(v,y)$, we have ${\rm d}(v,u)=q-2$, and hence ${\rm d}(x,u)\le q-1$. It follows that $u\in C_q(x,y)$. Therefore, there are exactly $c_{q-1}$ vertices in $C_q(x,y)$ at distance $q-2$ from $v$. By symmetry, for any $u\in C_q(x,y)$, there are exactly $c_{q-1}$ vertices in $C_q(y,x)$ at distance $q-2$ from $u$.

Construct a bipartite graph $H_1$ with bipartition $\{ C_q(y,x), C_q(x,y)\}$. For $v\in C_q(y,x)$ and $u\in C_q(x,y)$, $v$ and $u$ are adjacent if ${\rm d}(v,u)=q-2$. Then, $H_1$ is $c_{q-1}$-regular. Theorem \ref{konig} implies a desired bijection.
\end{proof}

\begin{lemma}
There is a bijection $\varphi$ from $A_q(y,x)$ to $A_q(x,y)$ such that ${\rm d}(v,\varphi(v))=q-1$ for every $v\in A_q(y,x)$.
\end{lemma}
\begin{proof}
For any $v\in A_q(y,x)$, we have ${\rm d}(v,y)=q$. We claim that $C_{q}(v,y)\subset A_q(x,y)$. Indeed, for any $u\in C_{q}(v,y)$, we have ${\rm d}(v,u)=q-1$, and hence ${\rm d}(x,u)\le q$. It follows that $u\notin B_q(x,y)$. If $u\in C_q(x,y)$, then $v\in A_{q-1}(u,x)$, which is contradictory to $|A_{q-1}(u,x)|=a_{q-1}=0$. Thus we have $u\in A_q(x,y)$ and the claim is proved. Therefore, there are exactly $c_{q}$ vertices in $A_q(x,y)$ at distance $q-1$ from $v$. By symmetry, for any $u\in A_q(x,y)$, there are exactly $c_{q}$ vertices in $A_q(y,x)$ at distance $q-1$ from $u$.

Similarly, we construct a bipartite graph $H_2$ with bipartition $\{ A_q(y,x), A_q(x,y)\}$. For $v\in A_q(y,x)$ and $u\in A_q(x,y)$, $v$ and $u$ are adjacent if ${\rm d}(v,u)=q-1$. Then, $H_2$ is $c_{q}$-regular. Theorem \ref{konig} implies a desired bijection.
\end{proof}
\begin{proof}[Proof of Theorem \ref{WassersteinDRG}] If $q\ge 2$, we construct a bipartite multigraph $H_3$ with bipartition 
$$\{ A_q(y,x)\cup B_q(y,x), A_q(x,y)\cup B_q(x,y)\}.$$
The edge set of $H_3$ is given by $E_H=E_1\cup E_2$, where
\begin{align}\notag
&E_1=\{vu|v\in A_q(y,x)\cup B_q(y,x), u\in A_q(x,y)\cup B_q(x,y), {\rm d}(v,u)=q\},\\ \notag
&E_2=\{e_v^j|e_v^j=v\varphi(v), v\in A_q(y,x), 1\le j\le c_{q+1}-a_q\}.
\end{align}
We explain that $E_2$ contains $c_{q+1}-a_q$ number of parallel edges between $v$ and $\varphi(v)$ for each $v\in A_q(y,x)$, and $E_2=\emptyset$ when $c_{q+1}= a_q$.

We claim that $H_3$ is $(c_{q+1}-c_q)$-regular.
For any $v\in B_q(y,x)$, we have ${\rm d}(v,y)=q+1$. There are exactly $c_{q+1}$ vertices in $S_1(y)$ at distance $q$ from $v$. For any $u\in C_q(x,y)$, since ${\rm d}(x,u)=q-1$, we have ${\rm d}(v,u)\le q$. Since ${\rm d}(v,y)=q+1$, we have ${\rm d}(v,u)\ge q$. It follows that ${\rm d}(v,u)= q$. Thus, there are exactly $c_{q+1}-c_q$ vertices in $A_q(x,y)\cup B_q(x,y)$ at distance $q$ from $v$. That is, the valency of $v$ in $H_3$ is $c_{q+1}-c_q$.

For any $v\in A_q(y,x)$, we have ${\rm d}(v,y)=q$. There are exactly $a_q$ vertices in $S_1(y)$ at distance $q$ from $v$. For any $u\in C_q(x,y)$, since ${\rm d}(x,u)=q-1$, we have ${\rm d}(v,u)\le q$. Since ${\rm d}(v,y)=q$, we have ${\rm d}(v,u)\ge q-1$. If ${\rm d}(v,u)= q-1$, then $v\in A_{q-1}(u,x)$, which is contradictory to $|A_{q-1}(u,x)|=a_{q-1}=0$. Thus, ${\rm d}(v,u)= q$. It follows that there are exactly $a_q-c_q$ vertices in $A_q(x,y)\cup B_q(x,y)$ at distance $q$ from $v$. Together with the $c_{q+1}-a_q$ parallel edges in $E_2$, the valency of $v$ in $H_3$ is $c_{q+1}-c_q$.

By symmetry, the valency of each vertex in $A_q(x,y)\cup B_q(x,y)$ is also $c_{q+1}-c_q$. Thus, $H_3$ is $(c_{q+1}-c_q)$-regular, as claimed.

By Theorem \ref{konig}, $E_H$ can be decomposed into $(c_{q+1}-c_q)$ edge-disjoint perfect matchings. Since $|E_2|=a_q(c_{q+1}-a_{q})$, there is a perfect matching $\mathcal M$ such that 
\begin{equation*}\label{ME}
    |\mathcal M\cap E_2|\ge M:= \left\lceil \frac{a_q(c_{q+1}-a_{q})}{c_{q+1}-c_q} \right\rceil.
\end{equation*}

We consider the following particular transport plan $\pi_0$ from $\mu_x^{\frac{1}{k+1}}$ to $\mu_y^{\frac{1}{k+1}}$:
\begin{center}
$\pi_0(v,u)=\begin{cases}
\frac{1}{k+1}, &{\rm if}\ v=x, u=y;\\
\frac{1}{k+1}, &{\rm if}\ v\in C_q(y,x), u=\phi(v);\\
\frac{1}{k+1}, &{\rm if}\ v\in A_q(y,x)\cup B_q(y,x), u\in A_q(x,y)\cup B_q(x,y), vu\in {\mathcal M};\\
0, &{\rm otherwise}.
\end{cases}$
\end{center}

It is direct to check that $\pi_0$ is indeed a transport plan. By the definition of $\phi$, we have ${\rm d}(v,\phi(v))=q-2$ for any $v\in C_q(y,x)$. For any $vu\in {\mathcal M}$, it follows by the definition of $E_1$ and $E_2$ that ${\rm d}(v,u)$ equals to $q$ if $vu\in E_1$ and $q-1$ if $vu\in E_2$. Therefore, we have
\begin{align}\notag
W\left(\mu_x^{\frac{1}{k+1}},\mu_y^{\frac{1}{k+1}}\right)
&\le \sum_{v\in V}\sum_{u\in V}{\rm d}(v,u)\pi_0(v,u)\\ \notag
&= \frac{1}{k+1}\left( q+c_q(q-2)+ |{\mathcal M}\cap E_2|(q-1) +(|{\mathcal M}|-|{\mathcal M}\cap E_2|)q \right)\\ \notag
&=\frac{1}{k+1}\left( q+c_q(q-2) -|{\mathcal M}\cap E_2|+ |{\mathcal M}|q \right)\\ \notag
&\le \frac{1}{k+1}\left( q+c_q(q-2) - M +(a_q+b_q)q \right)\\ \notag
&=q-\frac{2c_q + M}{k+1}.
\end{align}

If $q=1$, then $A_q(y,x)=A_q(x,y)$. This case has been discussed in \cite[Proof of Theorem 3.1]{CHLZ24}. For readers' convenience, we present the argument here. 
Let us denote the $a_1$ vertices in $A_1(y,x)$ by $z_1, \cdots, z_{a_1}$.
We construct a bipartite multigraph $H_4$ with bipartition 
$$\{A_1(y,x)\cup B_1(y,x), A'_1(x,y)\cup B_1(x,y)\}.$$
Here $A'_1(x,y):=\{z'_1, \cdots, z'_{a_1}\}$ is a new added set with ${a_1}$ vertices, which is considered as a copy of $A_1(y,x)$. The edge set of 
$H_4$ is given by $E_H:=\cup_{i=1}^5E_i$, where
\begin{align}\notag
&E_1=\{vu|v\in B_1(y,x), u\in B_1(x,y), v\sim u\},\\ \notag
&E_2=\{vz'_i|v\in B_1(y,x), z'_i\in A'_1(x,y), v\sim z_i\},\\ \notag
&E_3=\{z_iu| z_i\in A_1(y,x),u\in B_1(x,y), z_i\sim u\},\\ \notag
&E_4=\{z_iz'_j| z_i\sim z_j,1\le i\le a_1,1\le j\le a_1 \},\\ \notag
&E_5=\{e_i^j|e_i^j=z_iz'_i, 1\le i\le a_1, 1\le j\le c_2-a_1\}.
\end{align}
Similarly, we can prove that $H_4$ is $(c_2-c_1)$-regular. By Theorem \ref{konig}, $E_H$ can be decomposed into $c_2-c_1$ edge-disjoint perfect matchings. Since $|E_5|=a_1(c_{2}-a_{1})$, there is a perfect matching $\mathcal M$ such that 
\begin{equation*}
    |\mathcal M\cap E_5|\ge M:= \left\lceil \frac{a_1(c_{2}-a_{1})}{c_{2}-c_1} \right\rceil.
\end{equation*}
We consider the following particular transport plan $\pi_0$ from $\mu_x^{\frac{1}{k+1}}$ to $\mu_y^{\frac{1}{k+1}}$:
\begin{center}
$\pi_0(v,u)=\begin{cases}
\frac{1}{k+1}, &{\rm if}\ v\in B_1(y,x)\cup A_1(y,x), u\in B_1(x,y)\ {\rm and}\ vu\in {\mathcal M};\\
\frac{1}{k+1}, &{\rm if}\ v\in B_1(y,x)\cup A_1(y,x),u\in A_1(x,y)\ {\rm and}\ vu'\in {\mathcal M};\\
0, &{\rm otherwise}.
\end{cases}$
\end{center}
It is direct to check that $\pi_0$ is indeed a transport plan. There are $|{\mathcal M}|$ pairs of $(v,u)$ such that $\pi_0(v,u)\ne 0$. Among them, there are $|{\mathcal M}\cap E_5|$ pairs with ${\rm d}(v,u)=0$ and $|{\mathcal M}|-|{\mathcal M}\cap E_5|$ pairs with ${\rm d}(v,u)=1$. Therefore, we have
\begin{align}\notag
W\left(\mu_x^{\frac{1}{k+1}},\mu_y^{\frac{1}{k+1}}\right)&\le \sum_{v\in V}\sum_{u\in V}{\rm d}(v,u)\pi_0(v,u)\\ \notag
&= \frac{1}{k+1}(|{\mathcal M}|-|{\mathcal M}\cap E_5|)\\ \notag
&\le \frac{1}{k+1}\left(a_1+b_1-M\right)\\ \notag
&= 1-\frac{2+M}{k+1}.
\end{align}
We complete the proof.
\end{proof}

Now, we are prepared to prove Theorem \ref{diameter2} and Theorem \ref{diameter}.

\begin{proof}[Proof of Theorem \ref{diameter2}]
    For any two vertices $x,y$ with ${\rm d}(x,y)=p$,    
    Theorem \ref{WassersteinJump} shows that
    \begin{align}\label{JumpDRG}
        W\left(\mu_x^{1}, \mu_y^{\frac{1}{k+1}}\right)\le p+\frac{b_p-c_p}{k+1}.
    \end{align}
    The result then follows by Theorem \ref{Wasserstein} and Theorem \ref{WassersteinDRG}.
\end{proof}

\begin{proof}[Proof of Theorem \ref{diameter}]
    There exist two integers $l$ and $r$ with $l\ge 0$ and $0\le r\le q-1$ such that $d-p=lq+r$. Let $x$ and $y$ be two vertices with ${\rm d}(x,y)=d$. Let $L$ be a path of length $d$ connecting $x$ and $y$. On the path $L$, there is a sequence of vertices $x_0,x_1,...,x_{l}$ such that ${\rm d}(x,x_0)=p$, ${\rm d}(x_{i-1},x_i)=q$ for $1\le i\le l$, and ${\rm d}(x_{l},y)=r$.

    It follows by the triangle inequality that
    $$W_1\left(\mu_x^1,\mu_y^1\right) \le W_1\left(\mu_x^1,\mu_{x_0}^{\frac{1}{k+1}}\right) + \sum_{i=1}^l W_1\left(\mu_{x_{i-1}}^\frac{1}{k+1},\mu_{x_{i}}^\frac{1}{k+1}\right) + W_1\left(\mu_{x_l}^\frac{1}{k+1}, \mu_{y}^1\right).$$

    Note that $W_1\left(\mu_x^1,\mu_y^1\right)=d$. The inequality \eqref{JumpDRG} and Theorem \ref{WassersteinDRG} implies that
$$d\le \left( p+\frac{b_p-c_p }{k+1} \right)+ l\left(q-\frac{2c_q + M}{k+1}\right)+\left( r+\frac{b_r-c_r }{k+1} \right).$$

That is $$l\le \left\lfloor \frac{b_p-c_p+b_r-c_r}{2c_q+M} \right\rfloor,$$
completing the proof.
\end{proof}

\section{Further applications}
Our method applies not only to distance-regular graphs, but also to more general settings.  In this section, we take amply regular graphs and $(s,c,a,k)$-graphs for example.

\begin{definition}[Amply regular graph \cite{BCN89}] \label{ARGdefinition} 
    Let $G$ be a $k$-regular graph with $v$ vertices. Then $G$ is called an amply regular graph with parameters $(v,k,\lambda,\mu)$ if any two adjacent vertices have $\lambda$ common neighbors, and any two vertices at distance $2$ have $\mu$ common neighbors. 
\end{definition}

\begin{proof}[Proof of Theorem \ref{ARG}]
    For any two vertices $x,y$ with ${\rm d}(x,y)=2$,    
    Theorem \ref{WassersteinJump} shows that
        $$W\left(\mu_x^{1}, \mu_y^{\frac{1}{k+1}}\right)\le 2+\frac{|B_2(x,y)|-\mu}{k+1}\le 2+\frac{k-2\mu}{k+1}.$$
    For any two adjacent vertices $x$ and $y$, the same proof as Theorem \ref{WassersteinDRG} with $q=1$ shows that $$W\left(\mu_x^{\frac{1}{k+1}}, \mu_y^{\frac{1}{k+1}}\right)\le 1-\frac{2 + \left\lceil \frac{\lambda(\mu-\lambda)}{\mu-1} \right\rceil}{k+1}.$$
    The desired result then follows by Theorem \ref{Wasserstein}.
\end{proof}

\begin{definition}[$(s,c,a,k)$-graph \cite{T83}] \label{scakdefinition}
    Let $s,c,a$ and $k$ be integers with $s,c,a+2,k\ge 2$.
    An $(s,c,a,k)$-graph is a graph of maximum valence $k$ and girth $2s-1$ or $2s$ such that
    \begin{itemize}
        \item[\rm (1)] $|C_s(x,y)|=c$ for any two vertices $x,y$ with ${\rm d}(x,y)=s$,
        \item[\rm (2)] $|A_{s-1}(x,y)|=a$ for any two vertices $x,y$ with ${\rm d}(x,y)=s-1$.
    \end{itemize}
\end{definition}

\begin{lemma}[{\cite[Lemma 3.2]{T83}}]\label{scak83}
    An $(s, c, a, k)$-graph is either regular or bipartite, with all vertices in each partition having the same valency. In addition, $k_u=k_v$ for any two vertices $u,v$ with ${\rm d}(u,v)=s-1$.
\end{lemma}

\begin{proof}[Proof of Theorem \ref{scak}]
    For any two vertices $x,y$ with ${\rm d}(x,y)=s$ and $k_y=\delta$, Theorem \ref{WassersteinJump} shows that
        $$W\left(\mu_x^{1}, \mu_y^{\frac{1}{\delta+1}}\right)\le s+\frac{|B_s(x,y)|-c}{\delta+1}\le s+\frac{\delta-2c}{\delta+1}.$$
    For any two vertices $x,y$ with ${\rm d}(x,y)=s-1$ and $k_x=k_y=\delta$, the same proof as Theorem \ref{WassersteinDRG} with $q=s-1$ shows that $$W\left(\mu_x^{\frac{1}{\delta+1}}, \mu_y^{\frac{1}{\delta+1}}\right)\le s-1-\frac{2 + M}{\delta+1},\ {\rm where}\ M=\left\lceil \frac{a(c-a)}{c-1} \right\rceil.$$
    If $G$ is regular, the result follows by Theorem \ref{Wasserstein}. Otherwise, by Lemma \ref{scak83}, we suppose that $G$ is bipartite with bipartition $\{A,B\}$ such that each vertex in $A$ has valency $\delta$ and each vertex in $B$ has valency $k$. In addition, $k_u=k_v$ for any two vertices $u,v$ with ${\rm d}(u,v)=s-1$ implies that $s$ is odd.

    If \eqref{eq:scakdiameter} does not hold, there exist
    two vertices $x$ and $y$ with ${\rm d}(x,y)=D$ and $x\in B$ such that $D=2s+l(s-1)$, where $l$ is an integer satisfying
    \begin{align}\label{l2}
        l\ge 0\ {\rm and}\ l\ge \left\lfloor\frac{2(\delta-2c)}{2+M} \right\rfloor +1.
    \end{align}
    Let $L$ be a path of length $D$ connecting $x$ and $y$. On the path $L$, there is a sequence of vertices $x_0,x_1,...,x_{l}$ such that ${\rm d}(x,x_0)={\rm d}(x_{l} ,y)=s$ and ${\rm d}(x_{i-1},x_i)=s-1$ for $1\le i\le l$. Then, $x_i\in A$ for $0\le i\le l$.

    Note that $W_1(\mu_x^1,\mu_y^1)=D$. It follows by the triangle inequality that
    \begin{align}\notag
    D=W_1\left(\mu_x^1,\mu_y^1\right) &\le W_1\left(\mu_x^1,\mu_{x_0}^{\frac{1}{\delta+1}}\right) + \sum_{i=1}^l W_1\left(\mu_{x_{i-1}}^\frac{1}{\delta+1},\mu_{x_{i}}^\frac{1}{\delta+1}\right) + W_1\left(\mu_{x_l}^\frac{1}{\delta+1}, \mu_{y}^1\right)
    \\ \notag &\le 2\left(s+\frac{\delta-2c}{\delta+1}\right)+l\left(s-1-\frac{2 + M}{\delta+1}\right).
    \end{align}
    That is $l(2+M)\le 2(\delta-2c)$, which is contradictory to \eqref{l2}.
\end{proof}

\section{Acknowledgement}
 This work is supported by the National Key R \& D Program of China 2023YFA1010200 and the National Natural Science Foundation of China No. 12031017 and No. 12431004.



\begin{thebibliography}{99}

\bibitem{BDKM} S. Bang, A. Dubickas, J. H. Koolen and V. Moulton, There are only finitely many distance-regular graphs of fixed valency greater than two, Adv. Math. 269 (2015), 1–55.
\bibitem{BHK} S. Bang, A. Hiraki and J. H. Koolen, Improving diameter bounds for distance-regular graphs, European J. Combin. 27 (2006), 79–89,
\bibitem{BCN89} A. E. Brouwer, A. M. Cohen and A. Neumaier, Distance-regular graphs, Springer-Verlag, 1989.
\bibitem{CHLZ24} K. Chen, C. Hu, S. Liu and H. Zhang, Ricci curvature, diameter and eigenvalues of amply regular graphs, arXiv: 2410.21055, 2024.
\bibitem{CK19} D. Cushing, S. Kamtue, Long-scale Ollivier Ricci curvature of graphs, Anal. Geom. Metr. Spaces 7 (2019), no. 1, 22-44.
\bibitem{DKT} E. van Dam, J. H. Koolen and H. Tanaka, Distance-regular graphs, Dynamic Surveys, Electron. J.Combin., 2016, \\
http://www.combinatorics.org/ojs/index.php/eljc/article/view/DS22/pdf.
\bibitem{HLX24} X. Huang, S. Liu and Q. Xia, Bounding the diameter and eigenvalues of amply regular graphs via Lin--Lu--Yau curvature, Combinatorica (2024), 44 (2024), no. 6, 1177-1192.
\bibitem{I83} A. A. Ivanov, Bounding the diameter of a distance-regular graph, Dokl. Akad. Nauk SSSR 271(1983), 789–792.
\bibitem{K16} D. König, Über Graphen und ihre anwendung auf determinanten theorie und mengenlehre, Math. Ann. 77(1916), 453-465.
\bibitem{LLY11} Y. Lin, L. Lu and S.-T. Yau,  Ricci curvature of graphs, Tohoku Math. J. 63 (2011), no. 4, 605-627.
\bibitem{Mulder79} M. Mulder, $(0,\lambda)$-graphs and $n$-cubes, Discrete Math. 28 (1979), 179-188.
\bibitem{NP22JCTB} A. Neumaier and S. Penji\'c, A unified view of inequalities for distance-regular graphs, part I, J. Comb. Theory Ser. B 154 (2022), 392-439.
\bibitem{NP22} A. Neumaier and S. Penji\'c, On bounding the diameter of a distance-regular graph, Combinatorica 42 (2022), no. 2, 237-251.
\bibitem{O09} Y. Ollivier, Ricci curvature of Markov chains on metric spaces, J. Funct. Anal. 256 (2009), no. 3, 810-864.
\bibitem{Smith74} D. H. Smith, Bounding the diameter of a distance-transitive graph, J. Comb. Theory Ser. B 16 (1974), 139-144.
\bibitem{T82} P. Terwilliger, The diameter of bipartite distance-regular graphs, J. Comb. Theory Ser. B 32 (1982), 182-188.
\bibitem{T83} P. Terwilliger, Distance-regular graphs and $(s,c,a,k)$-graphs, J. Comb. Theory Ser. B 34 (1983), 151-164.




\end{thebibliography}
\end{document}